\documentclass{amsart}
\usepackage[utf8]{inputenc}

\input xy
\xyoption{all}

\usepackage{amssymb}
\usepackage{hyperref}
\usepackage[nospace,noadjust]{cite}
\usepackage[final]{showkeys}

\newtheorem{same}{This should never appear}[section]
\newtheorem{defin}[same]{Definition}

\newtheorem{remark}[same]{Remark}

\newtheorem{thm}[same]{Theorem}
\newtheorem{example}[same]{Example}

\newtheorem{lem}[same]{Lemma}
\newtheorem{fact}[same]{Fact}
\newtheorem{question}[same]{Question}
\newtheorem{cor}[same]{Corollary}

\newbox\noforkbox \newdimen\forklinewidth
\setbox0\hbox{$\textstyle\smile$}
\setbox1\hbox to \wd0{\hfil\vrule width \forklinewidth depth-2pt
 height 10pt \hfil}
\setbox\noforkbox\hbox{\lower 2pt\box1\lower 2pt\box0\relax}
\def\unionstick{\mathop{\copy\noforkbox}\limits}

\def\nonfork_#1{\unionstick_{\textstyle #1}}

\setbox0\hbox{$\textstyle\smile$}
\setbox1\hbox to \wd0{\hfil{\sl /\/}\hfil}
\setbox2\hbox to \wd0{\hfil\vrule height 10pt depth -2pt width
              \forklinewidth\hfil}
\newbox\doesforkbox
\setbox\doesforkbox\hbox{\lower 2pt\box1 \lower 2pt\box2\lower2pt\box0\relax}
\def\nunionstick{\mathop{\copy\doesforkbox}\limits}

\def\fork_#1{\nunionstick_{\textstyle #1}}

\newcommand{\LS}{\operatorname{LS}}

\newcommand{\ba}{\mathbf{a}}
\newcommand{\bb}{\mathbf{b}}

\newcommand{\K}{\mathcal{K}}

\newcommand{\leap}[1]{\le_{#1}}
\newcommand{\lea}{\leap{\K}}

\newcommand{\bx}{\mathbf{x}}
\newcommand{\by}{\mathbf{y}}
\newcommand{\bz}{\mathbf{z}}

\newcommand{\rest}{\upharpoonright}

\newcommand{\Ll}{\mathbb{L}}
\newcommand{\Lstruct}[1]{\Ll^{#1\text{-struct}}}
\newcommand{\Qstruct}{Q^{\text{struct}}}

\newcommand{\Mod}{\operatorname{Mod}}

\newcommand{\ran}{\operatorname{ran}}

\newcommand{\Ii}{\mathbb{I}}

\newcommand{\seq}[1]{\langle #1 \rangle}

\newcommand{\bL}{\mathbb{L}}

\newcommand{\cC}{\mathcal{C}}
\newcommand{\cD}{\mathcal{D}}
\newcommand{\cF}{\mathcal{F}}
\newcommand{\cL}{\mathcal{L}}

\newcommand{\comment}[1]{}

\newcommand{\ccl}{{\it c}\ell}

\newcommand{\bA}{\mathbf{A}}
\newcommand{\fct}[2]{{}^{#1}#2}
\newcommand{\Dd}{\mathbf{D}}
\newcommand{\PC}{\operatorname{PC}}

\title{Structural logic and abstract elementary classes with intersections}
\author{Will Boney}
\email{wboney@math.harvard.edu}
\urladdr{http://math.harvard.edu/\textasciitilde wboney/}
\address{Department of Mathematics, Harvard University, Cambridge, Massachusetts, USA}
\thanks{This material is based upon work done while
the first author was supported by the National Science Foundation under Grant No.\ DMS-1402191.}

\author{Sebastien Vasey}
\email{sebv@math.harvard.edu}
\urladdr{http://math.harvard.edu/\textasciitilde sebv/}
\address{Department of Mathematics, Harvard University, Cambridge, Massachusetts, USA}

\date{\today}

\begin{document}

\begin{abstract}
  We give a syntactic characterization of abstract elementary classes (AECs) closed under intersections using a new logic with a quantifier for isomorphism types that we call structural logic: we prove that AECs with intersections correspond to classes of models of a universal theory in structural logic. This generalizes Tarski's syntactic characterization of universal classes. As a corollary, we obtain that any AEC closed under intersections with countable Löwenheim-Skolem number is axiomatizable in $\Ll_{\infty, \omega} (Q)$, where $Q$ is the quantifier ``there exists uncountably many''.
\end{abstract}

\maketitle

\tableofcontents

\section{Introduction}

\subsection{Background and motivation}

Shelah's abstract elementary classes (AECs) \cite{sh88, baldwinbook09, shelahaecbook, shelahaecbook2} are a semantic framework to study the model theory of classes that are not necessarily axiomatized by an $\Ll_{\omega, \omega}$-theory. Roughly speaking (see Definition \ref{aec-def}), an AEC is a pair $(K, \lea)$ satisfying some of the category-theoretic properties of $(\Mod (T), \preceq)$, for $T$ an $\Ll_{\omega, \omega}$-theory. This encompasses classes of models of an $\Ll_{\infty, \omega}$ sentence (i.e.\ infinite conjunctions and disjunctions are allowed), and even $\Ll_{\infty, \omega} (\seq{Q_{\lambda_i}}_{i < \alpha})$ theories, where $Q_{\lambda_i}$ is the quantifier ``there exists $\lambda_i$-many''.

Since the axioms of AECs do not contain any axiomatizability requirement, it is not clear whether there is a natural logic whose class of models are exactly AECs. More precisely, one can ask whether there is a natural abstract logic (in the sense of Barwise, see the survey \cite{model-theoretic-logics}) so that classes of models of theories in that logic are AECs and any AEC is axiomatized by a theory in that logic.

An example of the kind of theorem one may expect is Tarski's characterization of universal classes. Tarski showed \cite{tarski-th-models-i} that classes of structures in a finite relational vocabulary which are closed under isomorphism, substructures, and union of chains (according to the substructure relation) are exactly the classes of models of a universal $\Ll_{\omega, \omega}$ theory. The proof of Tarski's result generalizes to non-finite vocabularies as follows:

\begin{defin}
  $K$ is a \emph{universal class} if it is a class of structures in a fixed vocabulary that is closed under isomorphisms, substructures, and unions of chains (according to the substructure relation).
\end{defin}

\begin{fact}[Tarski's presentation theorem, \cite{tarski-th-models-i}]\label{tarski-fact}
  Let $K$ be a class of structures in a fixed vocabulary. The following are equivalent:

  \begin{enumerate}
  \item $K$ is a universal class.
  \item $K$ is the class of models of a universal $\Ll_{\infty, \omega}$theory.
  \end{enumerate}
\end{fact}

Here, a universal sentence is one of the form $\forall x_0 \ldots x_{n - 1} \psi$ with $\psi \in \bL_{\infty, \omega}$ quantifier-free.  Note that this is not the only definition of universal sentences in the literature; see \cite[Remark 2.5]{categ-universal-2-selecta} for further discussion. 
 Universal classes are a special type of abstract elementary classes. In a sense, their complexity is quite low and indeed several powerful theorems can be proven there (see e.g.\ \cite[Chapter V]{shelahaecbook2} and the second author's ZFC proof of the eventual categoricity conjecture there \cite{ap-universal-apal, categ-universal-2-selecta}).

A more general kind of AECs are AECs with intersections. They were introduced by Baldwin and Shelah \cite[Definition 1.2]{non-locality}. They are defined as the AECs in which the intersection of any set of $\K$-substructures of a fixed model $N$ is again a $\K$-substructure, see Definition \ref{intersec-def}. In universal classes, this property follows from closure under substructure so any universal class is an AEC with intersections. Being closed under intersections does not imply that the classes are easy to analyze, e.g.\, \cite{hs-example, bk-hs, non-locality, lc-tame-pams} provide examples of AECs closed under intersections that fail to be tame. Nevertheless, AECs with intersections are still less complex than general AECs. For example, the second author has shown that Shelah's eventual categoricity conjecture holds there assuming a large cardinal axiom \cite[Theorem 1.7]{ap-universal-apal}, whereas the conjecture is still open for general AECs.

\subsection{Tarski's presentation theorem for AECs with intersections}

In the present paper, we generalize Tarski's presentation theorem to AECs with intersections as follows: we introduce a new logic, $\Lstruct{\kappa}_{\infty, \omega}$, which is essentially $\Ll_{\infty, \omega}$ expanded by what we call structural quantifiers, 
and show that AECs with intersections are (essentially) exactly the class of models of a particular kind of theory--what we call a $\forall\Qstruct$-theory (Definition \ref{aq-def})
--in the logic $\Lstruct{\kappa}_{\infty, \omega}$. More precisely (Corollary \ref{charact-cor}), any $\forall\Qstruct$-theory in $\Lstruct{\kappa}_{\infty, \omega}$ gives rise to an AEC with intersections and, for any AEC with intersections, there is an expansion of its vocabulary with countably-many relation symbols so that the resulting class is axiomatized by a $\forall\Qstruct$-theory in $\Lstruct{\kappa}_{\infty, \omega}$.  Moreover the expansion is functorial (i.e.\ it induces an isomorphism of concrete category, see Definition \ref{funct-def}).

The idea of the proof is to code the isomorphism types of the set $\ccl^N (\ba)$, where $\ccl^N (\ba)$ denotes the intersections of all the $\K$-substructures of $N$ containing the finite sequences $\ba$. The logic $\Lstruct{\kappa}_{\infty, \omega}$ is expanded by a family of generalized quantifiers in the sense of Mostowski and Lindstr\"{o}m \cite{mostowski-quantifiers, lindstrom-quantifiers}. We also add quantifiers such as $\Qstruct_{(M_2, M_1)} x y \phi (x) \psi (y)$, which asks whether the solution sets $(A, B)$ of $(\phi, \psi)$ are isomorphic to $(M_2, M_1)$ (where of course the two isomorphisms must agree). This is crucial to code the ordering of the AEC. Our characterization also generalizes Kirby's result on the definability of Zilber's quasiminimal classes \cite[\S5]{quasimin}. Indeed it is easy to see (Remark \ref{basic-rmk}) that $\Lstruct{\aleph_1}_{\omega_1,\omega}$ is just $\Ll_{\omega_1, \omega} (Q)$ (where $Q$ is the quantifier ``there exists uncountably many'') and quasiminimal classes are in particular AECs with intersections, see \cite{quasimin-aec-afml}. Thus we obtain that any AEC with intersections and countable Löwenheim-Skolem-Tarski number is axiomatizable in $\Ll_{\infty, \omega} (Q)$, see Corollary \ref{l-q-cor}.

An immediate conclusion is that any AEC which admits intersections, has Löwenheim-Skolem-Tarski number $\aleph_0$, and has countably-many countable models, has a Borel functorial expansion (in the sense that its restriction to $\aleph_0$ can be coded by a Borel set of reals, see Corollary \ref{borel-cor}). This is further evidence for the assertion that AECs with intersections have low complexity and paves the way for the use of descriptive set-theoretic tools in the study of these classes (see \cite[Chapter I]{shelahaecbook}, \cite{almost-galois-stable, baldwin-larson-iterated}).

\subsection{Other approaches}

Rabin \cite{rabin-inter} syntactically characterizes $\Ll_{\omega, \omega}$ theories whose class of models is closed under intersections. The characterization is (provably) much more complicated than that of universal classes. This paper shows that by changing the logic we can achieve a much easier characterization.  In a recent preprint \cite{multipres-v4-toappear}, Lieberman, Rosick\'y, and the second author have shown that any AEC with intersections is a locally $\aleph_0$-polypresentable category. In particular, it is $\aleph_0$-accessible, and this implies that it is equivalent (as a category) to a class of models of an $\Ll_{\infty, \omega}$ sentence. We discuss these approaches in greater detail in Section \ref{eva-sec}.

\comment{It is however false that an AEC with intersections is ``directly'' axiomatizable in $\Ll_{\infty, \omega}$: consider the AEC whose models are equivalence relations all of whose classes are countably infinite, ordered by the relation ``equivalence classes do not grow''. Thus asking for a direct axiomatizability, or axiomatizability up to functorial expansion (as in this paper) is stronger. Note that for any AEC $\K$ there is a functorial expansion which is $\Ll_{\infty, \LS (\K)^+}$-axiomatizable \cite[3.9]{baldwin-boney} but it is not known whether any AEC is directly $\Ll_{\infty, \LS (\K)^+}$-axiomatizable. There are several other facts and open questions about axiomatizability of AECs up to equivalence of categories. See for example \cite[\S4]{multipres-v4-toappear}, \cite{beke-rosicky}, \cite{lieberman-categ}, or \cite{makkai-pare}.}

\subsection{Notation}

We denote the universe of a $\tau$-structure $M$ by $|M|$ and its cardinality by $\|M\|$. We use $\kappa^-$ to denote the predecessor of a cardinal $\kappa$: it is $\kappa_0$ if $\kappa = \kappa_0^+$ and $\kappa$ otherwise.

\subsection{Acknowledgments}

We thank John T.\ Baldwin, Marcos Mazari-Armida, and the referee for feedback which helped improve the presentation of this paper.

\section{Structural quantifiers}\label{struct-sec}

We define a new logic, $\Lstruct{\kappa}$. It consists of $\Ll_{\omega, \omega}$ with a family of quantifiers $\Qstruct_{M, \bA}$, which generalize the quantifier $Q^{I(M)}$ from \cite[\S4]{makowsky-shelah-stavi}. Compared to $Q^{I(M)}$, we allow multiple formulas and $\tau_0$-structures, with $\tau_0$ a sub-vocabulary of $\tau$.

\begin{defin}
  Let $\tau$ be a vocabulary and $\kappa$ be an infinite cardinal. We define the logic $\Lstruct{\kappa} (\tau)$ as follows:

\begin{enumerate}
	\item $\Lstruct{\kappa}$ is the smallest set closed under the following:
	\begin{enumerate}
		\item Atomic formulas;
		\item Negation;
		\item Binary conjunction and disjunction;
		\item Existential and universal quantification; and
		\item If $n < \omega$, $\phi (x, \bz)$ and $\seq{\psi_i (y_i, \bz) : i < n}$ are formulas in $\Lstruct{\kappa}(\tau)$, $\tau_0$ is a sub-vocabulary of $\tau$, $M$ is a $\tau_0$-structure with universe an ordinal strictly less than $\kappa$ and $\bA := \seq{A_i : i < n}$ are subsets of $|M|$, then letting $\by := \seq{y_i : i < n}$:
                  
		  $$\Qstruct_{M, \bA} x \by \phi (x, \bz) \seq{\psi_i(y_i, \bz) : i < n} \in \Lstruct{\kappa}(\tau)$$
                  
	\end{enumerate}
	\item Satisfaction $\vDash_{\Lstruct{\kappa}}$ is defined inductively as follows (we omit the subscript since it will always be clear from context):
	\begin{enumerate}
		\item As usual for the first-order operations.
		\item If $N$ is a $\tau$-structure, $n < \omega$, $\phi (x, \bz)$ and $\seq{\psi_i (y_i, \bz) : i < n}$ are formulas in $\Lstruct{\kappa}(\tau)$, $\tau_0$ is a sub-vocabulary of $\tau$, $M$ is a $\tau_0$-structure with universe an ordinal strictly less than $\kappa$ and $\bA := \seq{A_i : i < n}$ are subsets of $|M|$, then letting $\by := \seq{y_i : i < n}$, $N \models \Qstruct_{M, \bA} x \by \phi (x, \bb) \seq{\psi_i(y_i, \bb) : i < n}$ if and only if:

                  \begin{enumerate}
                  \item For all $i < n$, $N \models \forall x \left( \psi_i (x, \bb) \rightarrow \phi (x, \bb)\right)$.
                  \item There is a $\tau_0$-substructure $N_0$ of $N$ with universe $\phi (N, \bb)$, and there is an isomorphism $f$ from $N_0$ onto $M$ such that for all $i < n$, $f[\psi_i (N, \bb)] = A_i$.
                  \end{enumerate}
	\end{enumerate}
\end{enumerate}

We define variants such as $\Lstruct{\kappa}_{\infty, \omega}$ in the natural way.
\end{defin}

Several notational remarks are in order: first, the restriction of the universe of $M$ to be an ordinal is here to avoid having a logic with class-many formulas. We will often use $\Qstruct_{M, \bA}$ when the universe of $M$ is \emph{not} an ordinal, and this should just be replaced by $\Qstruct_{M', \bA'}$, where $M', \bA'$ are renaming of $M$ and $\bA$ so that the universe of $M'$ is an ordinal. 

Second, the main cases for us in the formula $\Qstruct_{M, \bA} x \by \phi (x, \bz) \seq{\psi_i(y_i, \bz) : i < n}$ are when $n = 0$ or $n = 1$. In the former case, we will just write $\Qstruct_M x \phi(x, \bz)$ and in the latter $\Qstruct_{M, A} x y \phi(x, \bz) \psi(y, \bz)$. In most cases, $A$ will induce a substructure of $M$.  

In order to analyze a class of models axiomatized by structural quantifiers, we use a strong notion of elementarity, $\preceq_\cF^*$, which ensures that classes of models of an $\Lstruct{\kappa}$-theory are closed under unions of $\preceq_{\cF}^*$-increasing chains.

\begin{defin} \
\begin{enumerate}
	\item Given a language $\tau$ and a logic $\cL$, a \emph{fragment} $\cF$ is a set $\cF \subseteq \cL(\tau)$ closed under subformula and containing all atomic formulas.
	\item Given a fragment $\cF \subseteq \Lstruct{\kappa}_{\infty, \omega}(\tau)$ and two $\tau$-structures $N_1 \subseteq N_2$, $N_1 \preceq_\cF N_2$ if and only if for every $\phi(\bx) \in \cF$ and $\ba \in N_1$,
	$$N_1\vDash \phi(\ba) \iff N_2 \vDash \phi(\ba)$$
	\item Given a fragment $\cF \subseteq \Lstruct{\kappa}_{\infty, \omega}(\tau)$ and two $\tau$-structures $N_1 \subseteq N_2$, $N_1 \preceq^*_\cF N_2$ if and only if:
	\begin{enumerate}
	\item $N_1 \preceq_\cF N_2$; and
        \item For any formula $\Qstruct_{M, \bA} x \by \phi (x, \bz) \seq{\phi_i (y_i, \bz) : i < n}$ in $\cF$, if $|\phi (N_1, \bb)| < \kappa$, then $\phi (N_1, \bb) = \phi (N_2, \bb)$ and $\psi_i (N_1, \bb) = \psi_i (N_2, \bb)$ for all $i < n$.
        \end{enumerate}
      \item Given a fragment $\cF \subseteq \Lstruct{\kappa}_{\infty, \omega}$ and a theory $T \subseteq \cF$, we let $\Mod_{\cF} (T)$ be the class $(\Mod (T), \preceq_{\cF}^*)$.
        
\end{enumerate}
\end{defin}

Note that we do \emph{not} assume that a fragment must be closed under the finitary operations, just that it contains the atomic formulas and is closed under subformulas.

We also have the following basic facts:

\begin{remark}\label{basic-rmk} \
  \begin{enumerate}
  \item $\Lstruct{\aleph_0}_{\lambda, \kappa}$ is equivalent to $\Ll_{\lambda, \kappa}$.
  \item For any ordinal $\alpha$, we can express $\exists^{\geq \aleph_\alpha} x \phi(x, \bz)$ in $\Lstruct{\aleph_\alpha}_{(|\alpha| + \aleph_0)^+, \omega}$ by:

    $$
    \bigwedge_{n < \omega} \neg \Qstruct_n x \phi (x, \bz) \land \bigwedge_{\beta < \alpha} \neg \Qstruct_{\aleph_\beta} x \phi(x, \bz)
    $$

    In particular ``there exists uncountably many'' can be expressed in $\Lstruct{\aleph_1}_{\omega_1, \omega}$.
  \item \label{lw1wq} In fact, $\Lstruct{\aleph_1}_{\lambda^+, \omega}$ is equivalent to $\Ll_{\lambda^+, \omega} (Q)$, where $Q$ is the quantifier ``there exists uncountably many''. We have just established the right to left direction. For the other direction, use suitably relativized Scott sentences. Specifically, let us explain how to replace an $\Lstruct{\aleph_1}_{\lambda^+, \omega}$ formula of the form $\Qstruct_{M, \bA} x \by \phi (x, \bz) \seq{\psi_i (y_i, \bz) : i < n}$ by an equivalent formula in $\Ll_{\lambda^+,\omega} (Q)$. Suppose that $\tau$ is the vocabulary of the formula and $\tau_0$ the vocabulary of $M$. Let $\tau_0' = \tau_0 \cup \{P_i : i < n\}$, where the $P_i$'s are unary predicates not appearing in $\tau$. Let $M_{\bA}$ be the $\tau_0'$ structure $(M, A_i)_{i < n}$, where we have written $\bA = \seq{A_i : i < n}$. Let $\rho$ be a Scott sentence for $M_{\bA}$. Let $\rho'$ be $\rho$ relativized to $\phi (\cdot, \bz)$, with for all $i < n$ all occurrences of $P_i (y)$ replaced by $\psi_i (y, \bz)$. Then $\Qstruct_{M, \bA} x \by \phi (x, \bz) \seq{\psi_i (y_i, \bz) : i < n}$ is equivalent to $\bigwedge_{i < n} \forall x (\psi_i (x, \bz) \to \phi (x, \bz)) \land \neg Q x \phi (x, \bz) \land \rho'$. 
  \item The notion of elementarity $\preceq^*_{\cF}$ introduced for structural quantifiers coincides with the notion of elementarity used to study $\bL(Q)$, see \cite[Definition 5.1.2]{baldwinbook09}.
  \item An $\Lstruct{\kappa} (\tau)$-formula of the form $\Qstruct_{M, \bA} \phi (x, \bz) \seq{\psi_i (y_i, \bz) : i < n}$, with $M$ a $\tau_0$-structure, $\tau_0 \subseteq \tau$ is equivalent to a disjunction of formulas of the form $\Qstruct_{M', \bA} \phi (x, \bz) \seq{\psi_i (y_i, \bz) : i < n}$ , where $M'$ ranges over all isomorphism types of $\tau$-expansion of $M$. Thus we can avoid the use of a sub-vocabulary $\tau_0$ but have to consider potentially longer (of length up to $2^{|\tau|}$) disjunctions.
  \end{enumerate}
\end{remark}

We now show that $\Mod_{\cF} (T)$ defined above is an abstract elementary class (AEC). For the convenience of the reader, we repeat the definition of an AEC here.

\begin{defin}[\cite{sh88}]\label{aec-def}
  An \emph{abstract elementary class (AEC)} is a pair $\K = (K, \lea)$ satisfying the following properties:

  \begin{enumerate}
  \item $K$ is a class of structures in a fixed vocabulary $\tau = \tau (\K)$ and $\lea$ is a partial order, $M \lea N$ implies $M \subseteq N$, and $K$ and $\lea$ are both closed under isomorphisms.
  \item \underline{Coherence:} if $M_0, M_1, M_2 \in \K$, $M_0 \subseteq M_1 \lea M_2$ and $M_0 \lea M_2$, then $M_0 \lea M_1$.
  \item \underline{Tarski-Vaught chain axioms:} if $\delta$ is a limit ordinal, $\seq{M_i : i < \delta}$ is a $\lea$-increasing chain, and $M := \bigcup_{i < \delta} M_i$, then:
    \begin{enumerate}
    \item $M \in \K$.
    \item $M_0 \lea M$.
    \item Smoothness: if $N \in \K$ is such that $M_i \lea N$ for all $i < \delta$, then $M \lea N$.
    \end{enumerate}
  \item \underline{L\"{o}wenheim-Skolem-Tarski (LST) axiom:} there exists a cardinal $\lambda \ge |\tau (\K)| + \aleph_0$ such that for any $N \in \K$ and any $A \subseteq |N|$, there exists $M \in \K$ such that $A \subseteq |M|$, $M \lea N$, and $\|M\| \le |A| + \lambda$. We write $\LS (\K)$ for the least such $\lambda$.
  \end{enumerate}
\end{defin}

\begin{thm}\label{general-aec-thm}
  Let $\cF$ be a fragment of $\Lstruct{\kappa}_{\infty, \omega} (\tau)$ and let $T \subseteq \cF$ be a theory. Then $\K := \Mod_{\cF} (T)$ is an AEC with $\LS (\K) \le |\cF| + \kappa$.
\end{thm}
\begin{proof}
  This is a very similar situation to $\bL(Q)$ where $Q$ is ``there exists uncountably many'', and most the axioms are straightforward. We only show that $\K$ is closed under unions of increasing chains. Let $\seq{N_i \mid i < \delta}$ be continuous, $\preceq_\cF^*$-increasing and let $N_\delta := \bigcup_{j<\delta} N_j$. We want to show that $N_0 \preceq_{\cF}^\ast N_\delta$. We work by induction on formulas. The steps are standard except for the structural quantifier.  Let $\chi (\bz)$ be a formula in $\cF$ of the form $\Qstruct_{M, \bA} x \by \phi (x, \bz) \seq{\psi_i (y_i, \bz) : i < n}$.

  By elementarity of the inner formulas, we know that for $\psi \in \Psi := \{\phi\} \cup \{\psi_i : i < n\}$, $\psi(N_\delta, \bb) = \bigcup_{i<\delta} \psi (N_i, \bb)$.

  We consider two cases. If $|\phi (N_0, \bb)| < \kappa$, then by the definition of $\preceq_{\cF}^\ast$, $\psi (N_0, \bb) = \psi (N_i, \bb)$ for all $i < \delta$ and $\psi \in \Psi$. This implies that for all $\psi \in \Psi$, $\psi (N_0, \bb) = \psi (N_\delta, \bb)$, and so $N_0 \models \chi (\bb)$ if and only if $N_\delta \models \chi (\bb)$.

  The remaining case is when $|\phi (N_0, \bb)| \ge \kappa$. In this case, $N_0 \not \models \chi (\bb)$ (by definition of $\Lstruct{\kappa}$) and since $\phi (N_0, \bb) \subseteq \phi (N_\delta, \bb)$, we must also have that $|\phi (N_\delta, \bb)| \ge \kappa$ and hence $N_\delta \not \models \chi (\bb)$.

\end{proof}

The following example is due to Shelah (see the beginning of section 3 of \cite{sh88}) and has been further examined by Kueker \cite[Example 6.3]{kueker2008}:

\begin{example}
Fix an infinite cardinal $\lambda$.  Set $\K^\lambda$ to be the AEC consisting of well-orderings $(X, <)$ of either size $\lambda$ or of order type $\lambda^+$ with the ordering on the class $\K^\lambda$ being initial segment.  The rigidity of well-orderings implies that this is an AEC with intersections with $\LS(\K^\lambda) = \lambda$.  This class can be axiomatized in $\Lstruct{\lambda^+}_{\lambda^{++}, \omega}$ by the sentence
$$\forall x \bigvee_{\lambda \leq \alpha < \lambda^+} \Qstruct_{(\alpha, \in)} y (y < x)$$

Several natural logics whose classes of models form AECs are unsuccessful in axiomatizing this class (although Remark \ref{basic-rmk}.(\ref{lw1wq}) implies that $\K^{\aleph_0}$ is $\bL_{\omega_2, \omega}(Q)$ axiomatizable).
\end{example}

\begin{remark}
Although classes axiomatizable in $\Lstruct{\kappa}_{\infty,\omega}$ form AECs, it seems unlikely that every AEC is axiomatizable with structural quantifiers.  In particular, classes axiomatizable using Shelah's cofinality quantifier $Q^{cf}_\kappa$ (see \cite[Definition 1.3]{sh43}) form (with the right notion of strong substructure) an AEC.  However, it seems unlikely that such classes can be axiomatized using structural quantifiers as there is no way to pick out a witnessing sequence in the linear order even after functorial expansion.  By Theorem \ref{aec-to-logic}, this means that these classes are not closed under intersections.
\end{remark}

In the next section, we will be interested in sentences in structural logic of a particular form:

\begin{defin}\label{aq-def}
  A sentence of $\Lstruct{\kappa}_{\infty, \omega}$ is called a $\forall\Qstruct$-sentence if it is of the form:

  $$
  \forall \bz \bigvee_{(M, \bA) \in S} \Qstruct_{M, \bA} x \by \phi (x, \bz) \seq{\psi_i (y_i, \bz) : i < n}
  $$

  where $\phi$ and the $\psi_i$'s are quantifier-free and $S$ is a non-empty set. A theory in $\Lstruct{\kappa}_{\infty, \omega}$ is called a \emph{$\forall\Qstruct$-theory} if it consists exclusively of $\forall\Qstruct$-sentences.
\end{defin}

Note that this generalizes the usual definition of a universal sentence in the following sense:

\begin{lem}\label{univ-gen}
  If $\phi = \forall \bz \psi (\bz)$, where $\psi$ is quantifier-free in $\Ll_{\infty, \omega} (\tau)$, then there exists a $\forall\Qstruct$-formula $\phi'$ of $\Lstruct{\aleph_0}_{\infty, \omega} (\tau)$ such that $\phi$ and $\phi'$ have the same models.
\end{lem}
\begin{proof}
  Say $\bz = \seq{z_i : i < n}$ for $n < \omega$. If $n > 0$, take $\phi' = \forall \bz \Qstruct_{1} x (x = z_0 \land \psi (\bz))$, where we see $1$ as a structure in the empty language whose universe has one element. If $n = 0$, then there are two cases: either the empty structure satisfies $\phi$, in which case $\phi$ is equivalent to $\forall z_0 \phi$, and we can proceed as before, or the empty structure does \emph{not} satisfy $\phi$ which must mean (since $\phi$ is quantifier-free) that there are constant symbols in $\tau$. In this case the empty structure is never a $\tau$-structure, and so as before $\phi$ is equivalent to $\forall z_0 \phi$.
\end{proof}

One may argue that a $\forall\Qstruct$ sentence is more like a $\forall \exists$ sentence than just a universal sentence. However, Remark \ref{tarski-gen-rmk} highlights that the inner disjunction plays the role that a quantifier-free formula would play in the proof of Tarski's presentation theorem. Thus $\forall \Qstruct$ sentences play the same role that universal sentences play in $\Ll_{\infty, \omega}$.

\section{Axiomatizing abstract elementary classes with intersections}

Recall the definition of an AEC with intersections:

\begin{defin}[{\cite[Definition 1.2]{non-locality}}]\label{intersec-def}
  Let $\K$ be an AEC.

  \begin{enumerate}
  \item For $N \in \K$ and $A \subseteq |N|$, write $\ccl_{\K}^N (A)$ for the set $\bigcap \{N_0 : N_0 \lea N \land A \subseteq |N_0|\}$. Almost always, $\K$ is clear from context and we omit it. We abuse notation and also write $\ccl^N (A)$ for the $\tau (\K)$-substructure of $N$ with universe $\ccl^N (A)$. When $\ba \in \fct{<\infty}{M}$, we write $\ccl^N (\ba)$ for $\ccl^N (\ran (\ba))$.
  \item We say that $\K$ \emph{has intersections} (or \emph{admits intersections} or is an AEC \emph{with intersections}, or is \emph{closed under intersections}) if $\ccl^N (A) \lea N$ for any $N \in \K$ and $A \subseteq |N|$.
  \end{enumerate}
\end{defin}

We show that the class of models of a $\forall\Qstruct$-theory in structural logic is an AEC with intersections:

\begin{thm}\label{logic-to-aec}
  Let $T$ be a $\forall\Qstruct$-theory of $\Lstruct{\kappa}_{\infty, \omega}$ and let $\cF$ be the smallest fragment containing $T$. Then $\Mod_{\cF} (T)$ is an AEC with intersections and L\"{o}wenheim-Skolem-Tarski number at most $|\cF| + \kappa^-$.
\end{thm}
\begin{proof}
  Let $\K := \Mod_{\cF} (T)$. By Theorem \ref{general-aec-thm}, $\K$ is an AEC. We show that $\K$ has intersections. Let $N \in \K$, $A \subseteq |N|$, and let $N_0 := \ccl^N (A)$. We show that $N_0 \preceq_{\cF}^* N$. We proceed by induction on formulas. The atomic, conjunction, disjunction, and negation cases are trivial. The universal case also follows directly from our assumption on the theory (recall that the definition of fragment used here does \emph{not} require closure under the finitary operations).

  It remains to deal with the case of a formula $\chi_{M, \bA} (\bz)$ of the form $\Qstruct_{M, \bA} x \by \phi (x, \bz) \seq{\phi_i (y_i, \bz) : i < n}$. Let $\Psi := \{\phi \} \cup \{\psi_i : i < n\}$ and let $\bb \in N_0$. Assume first that $N \models \chi_{M, \bA} (\bb)$. Let $\psi \in \Psi$. For any $N'$ with $A \subseteq |N'|$ and $N' \preceq_{\cF}^\ast N$ we must have (by definition of $\ccl$) that $\bb \in N'$ and $\psi (N', \bb) = \psi (N, \bb)$. It follows that also $\psi (N_0, \bb) = \psi (N, \bb)$, so $N_0 \models \chi_{M, \bA} (\bb)$. Assume now that $N_0 \models \chi_{M, \bA} (\bb)$. By the assumption on the theory, there must exist a $(M', \bA')$ such that $N \models \chi_{M', \bA'} (\bb)$. By the previous argument, $N_0 \models \chi_{M', \bA'} (\bb)$, and moreover $\psi (N, \bb) = \psi (N_0, \bb)$ for any $\psi \in \Psi$. Since $N_0 \models \chi_{M, \bA} (\bb)$, this directly implies that $N \models \chi_{M, \bA} (\bb)$.

  By a similar argument, we must have (regardless of the truth value of $\chi_{M, \bA}$ in $N_0$ or $N$) that for any $\psi \in \Psi$, $|\psi (N, \bb)| < \kappa$ and $\psi (N_0, \bb) = \psi (N, \bb)$. It also follows that $|\ccl^N (A)| \le |A| + |\cF| + \kappa^-$, so $\LS (\K) \le |\cF| + \kappa^-$.
\end{proof}

We now work toward a converse. For this, we will use the notion of a functorial expansion \cite[Definition 3.1]{sv-infinitary-stability-afml}. This is a class in an expanded language that looks exactly the same as the original AEC. For example, the Morleyization of an elementary class is a functorial expansion. For more examples, see \cite[\S3]{sv-infinitary-stability-afml}.

\begin{defin}\label{funct-def}
  For $\K$ an AEC, a \emph{functorial expansion} of $\K$ is a class $K^+$ of structures in a vocabulary $\tau (K^+) \supseteq \tau (\K)$ (in this paper always finitary) such that the reduct map is structure-preserving bijection from $K^+$ onto $\K$. That is, if two structures are isomorphic in $\K$ then their expansions to $K^+$ are still isomorphic and if $M$ is a $\K$-substructure of $N$, then their expansions to $K^+$ are substructures. For a functorial expansion $K^+$, we let $\K^+ := (K^+, \leap{\K^+})$, where $M \leap{\K^+} N$ if and only if $M \rest \tau (\K) \lea N \rest \tau (\K)$. We will also say that $\K^+$ is a functorial expansion of $\K$.
\end{defin}

\begin{remark}\label{funct-rmk}
  An AEC admits intersections if and only if it admits intersections in some functorial expansion.
\end{remark}

In addition to the obvious monotonicity properties of $\ccl^N$, we will use the following basic fact about AECs with intersections:

\begin{fact}[{\cite[Proposition 2.14(4)]{ap-universal-apal}}]
  Let $\K$ be an AEC with intersections. Let $M \lea N$ and let $A \subseteq |M|$. Then $\ccl^M (A) = \ccl^N (A)$.
\end{fact}

In any AEC, one can define a notion of orbital types (also called Galois types in the literature), see \cite[II.1.9]{shelahaecbook}. We have no use for the general notion of orbital type in this paper, we will just recall what orbital types of finite sequences over the empty set look like in AECs admitting intersections:

\begin{defin}
  Let $\K$ be an AEC with intersections. For $M, N \in \K$ and $\ba \in \fct{<\omega}{M}$, $\bb \in \fct{<\omega}{N}$, we write $(\ba, M) \equiv (\bb, N)$ if there exists $f: \ccl^M (\ba) \cong \ccl^N (\bb)$ such that $f (\ba) = \bb$. Note that this is an equivalence relation. We denote by $\Dd (\K)$ the set of all $\equiv$-equivalence classes.
\end{defin}
\begin{remark}\label{small-repr}
  Let $\K$ be an AEC with intersections. If $M_0 \lea M$ are both in $\K$ and $\ba \in \fct{<\omega}{M}$, then $(\ba, M_0) \equiv (\ba, M)$ (as witnessed by the identity map). In particular, for any $(\bb, N) \in \Dd (\K)$, the Löwenheim-Skolem-Tarski axiom of AECs implies that there exists $N_0 \in \K_{\le \LS (\K)}$ such that $(\bb, N_0) \in \Dd (\K)$.
\end{remark}

\begin{lem}\label{embed-dk}
  Let $\K$ be an AEC with intersections. Then $|\Dd (\K)| \le 2^{\LS (\K)}$. More precisely, $|\Dd (\K)| < \LS (\K)^+ + \mu$, where $\mu$ is the least cardinal such that for any (not necessarily increasing) sequence $\seq{M_i : i < \mu}$ of elements of $\K_{\le \LS (\K)}$, there exists $i < j < \mu$ such that $M_j$ embeds into $M_i$.
\end{lem}
\begin{proof}
  Let $\theta := \mu + \LS (\K)^+$. Suppose for a contradiction that $|\Dd (\K)| \ge \theta$. We build $\seq{(\ba_i, M_i) : i < \theta}$ such that for all $i < \theta$:

  \begin{enumerate}
  \item $\ba_i \in \fct{<\omega}{M_i}$.
  \item $M_i \in \K_{\le \LS (\K)}$.
  \item There is no $i_0 < i$ and $\ba \in \fct{<\omega}{M_{i_0}}$ such that $(\ba, M_{i_0}) \equiv (\ba_i, M_i)$.
  \end{enumerate}

  This is possible by the assumption that $|\Dd (\K)| \ge \theta$ and $\theta > \LS (\K)$. This is enough: by the definition of $\mu$, there exist $i_0 < i < \theta$ such that $M_i$ embeds into $M_{i_0}$. This contradicts the assumption that $(\ba_i, M_i)$ had no realization in $M_{i_0}$.
\end{proof}

\begin{thm}\label{aec-to-logic}
  Let $\K$ be an AEC with intersections. Then there is
\begin{enumerate}
	\item A functorial expansion $\K^+$ of $\K$ with vocabulary $\tau^+= \tau (\K) \cup\{E_n \mid n < \omega\}$.
	\item A $\forall\Qstruct$-theory $T$ in $\Lstruct{\LS (\K)^+}_{\left(\LS (\K) + |\Dd (\K)|\right)^+, \omega}(\tau^+)$.
\end{enumerate}

such that $\K^+ = \Mod_{\cF} (T)$, where $\cF$ is the smallest fragment containing $T$.
\end{thm}
\begin{proof}
  We define $\K^+$ as follows:
  \begin{itemize}
    \item $E_n$ is an $(n + 1)$-ary relation symbol.
    \item $M^+ \in \K^+$ if and only if $M^+\rest \tau (\K) \in \K$ and $E_n^{M^+} = \{(a, \bb) \mid  \ell (\bb) = n, a \in \ccl^{M^+ \rest \tau (\K)}(\bb)\}$.
\end{itemize}

$T$ is the following theory:
\begin{enumerate}
	\item For each $n < \omega$ and each $k \le n$, include:
	  $$\forall z_0 \ldots z_{n} E_n (z_k, z_0, z_1, \ldots, z_n)$$
	Note that this sentence is not formally a $\forall\Qstruct$-sentence, but is equivalent to one by Lemma \ref{univ-gen}.
	\item If $\K$ is empty, include $\Qstruct_0 x x = x$ and $\Qstruct_1 x x \neq x$. Assume now that $\K$ is not empty. Let $D$ be a complete set of representatives of the elements of $\Dd (\K)$ (so $|D| = |\Dd (\K)|$), with the additional requirement that $(\ba, M) \in D$ implies $\|M\| \le \LS (\K)$ (this is possible by Remark \ref{small-repr}). Let $S$ be the set of pairs $(M_2, M_1)$ such that for some $(\ba \bb, M) \in D$, $M_2 = \ccl^M (\ba \bb)$, $M_1 = \ccl^M (\ba)$. Note that for such pairs $M_1 \lea M_2$ by coherence. Moreover $|S| \le |D| + \LS (\K)$. Include in $T$ all sentences of the form:
          
	$$\forall\bz,\bz' \bigvee_{(M_2, M_1) \in S} \Qstruct_{M_2, M_1} x y \left(E_{\ell (\bz \bz')} (x \bz \bz')\right) \left(E_{\ell (\bz)} (y \bz)\right)$$
\end{enumerate}
Let $\cF$ be the smallest fragment containing $T$.  We show that this works.
\begin{enumerate}
	\item Let $M^+ \in \K^+$ and let $M := M^+ \rest \tau (\K)$. We have:
	\begin{itemize}
		\item For all $\ba \in M$, $\ba \in \ccl^M(\ba)$; and
		\item For all $\ba_1, \ba_2\in M$,
                  
		$$M^+ \models \Qstruct_{\ccl^M(\ba_1\ba_2), \ccl^M(\ba_1)}x y \left(E_{\ell (\ba_1)} (x \ba_1)\right) \left(E_{\ell (\ba_1 \ba_2)} y \ba_1\ba_2\right)$$
	\end{itemize}
	So $M^+ \vDash T$.
	\item Suppose that $M^+ \leap{\K^+} N^+$.  Then $\ccl^M(\ba) = \ccl^N(\ba)$ are of size at most $\LS(\K)$, so $M^+\preceq^*_{\cF}N^+$ follows.
	
	\item Suppose that $M^+ \vDash T$.  We wish to express it as a directed colimit from $\K^+$.  For $\ba \in M^+$, let $M_\ba$ be the $\tau^+$-substructure of $M^+$ with universe $\{b \in M^+ \mid M^+\vDash E_{\ell (\ba)} (b \ba)\}$.  By definition of $T$ (and definition of a functorial expansion), $\ba \in M_{\ba}$, $M_\ba \in \K^+$, and $\ba \subseteq \bb$ implies that $M_\ba \leap{\K^+} M_\bb$.  Thus, 
	$$M^+ = \bigcup_{\ba \in M^+}M_\ba \in \K^+$$
	
	\item Suppose that $M^+, N^+ \models T$ and $M^+ \preceq_\cF^*N^+$.  This elementarity implies that, in the notation of the previous item, given $\ba \in M^+$, $M_\ba = N_\ba$. Now as before $M^+ = \bigcup_{\ba \in M^+} M_{\ba} = \bigcup_{\ba \in M^+} N_{\ba} \leap{\K^+} N^+$, as needed.
\end{enumerate}
\end{proof}
\begin{remark}\label{tarski-gen-rmk}
  It is instructive to see how the proof of Theorem \ref{aec-to-logic} plays out when $\K$ is a universal class. In this case, $E_n$ is definable, i.e.\ we can replace the formula $E_n (x, \by)$ by $\bigvee_{\rho} x = \rho (\by)$, where $\rho$ ranges over all terms in the vocabulary of $\K$. Thus there is no need to expand the vocabulary of $\K$. Similarly, if $(M_2, M_1)$ is a pair in $\K$ with $M_1 \subseteq M_2$ such that $M_2 = \ccl^{M_2} (\ba \bb)$, $M_1 = \ccl^{M_2} (\ba)$, then $\Qstruct_{M_2, M_1} x y (E_{\ell (\bz \bz')} (x, \bz \bz')) (E_{\ell (\bz)} (y \bz))$ is equivalent to $\bigwedge_{\phi \in \operatorname{tp}_{\text{qf}} (\bb_1 \bb_2; M_2)} \phi (\bz; \bz')$. Thus we indeed recover Tarski's characterization of universal classes as classes of models of a universal $\Ll_{\infty, \omega}$ theory.
\end{remark}

We have arrived to the promised characterization of AECs with intersections using structural logic:

\begin{cor}\label{charact-cor}
  Let $\K$ be an AEC. The following are equivalent:

  \begin{enumerate}
  \item $\K$ has intersections.
  \item There is a functorial expansion $\K^+$ of $\K$ with countably-many relation symbols of finite arity and a $\forall\Qstruct$-theory $T$ in $\Lstruct{\LS (\K)^+}_{\infty, \omega} (\tau (\K^+))$ such that $\K^+ = \Mod_{\cF} (T)$, where $\cF$ is the smallest fragment containing $T$.
  \end{enumerate}
\end{cor}
\begin{proof}
  Combine Remark \ref{funct-rmk}, and Theorems \ref{logic-to-aec}, \ref{aec-to-logic}.
\end{proof}

As a particular case, we get that AECs with intersections and countable Löwenheim-Skolem-Tarski number are axiomatizable in $\Ll_{\infty, \omega} (Q)$:

\begin{cor}\label{l-q-cor}
  Let $\K$ be an AEC with intersections and $\LS (\K) = \aleph_0$. Then there is a functorial expansion $\K^+$ with countably-many relation symbols of finite arity and an $\Ll_{\aleph_1 + |\Dd (\K)|^+, \omega} (Q)$-theory $T$ such that $\K^+$ is the class of models of $T$.
\end{cor}
\begin{proof}
  Combine Theorem \ref{aec-to-logic} with Remark \ref{basic-rmk}(\ref{lw1wq}).
\end{proof}
\begin{remark}
  We may no longer have that the ordering on $\K^+$ is given by elementarity according to a fragment of $\Ll_{\infty, \omega} (Q)$ containing $T$. This is because we are using Scott sentences to code isomorphism types. Nevertheless, the axiomatization given by Corollary \ref{l-q-cor} uses only negative instances of $Q$ so indeed leads to an AEC with countable Löwenheim-Skolem number.
\end{remark}

We state a consequence on the descriptive set-theoretic complexity of $\K$. Recall that Baldwin and Larson call an AEC $\K$ \emph{analytically presented} \cite[Definitions 3.1 and 3.3]{baldwin-larson-iterated} if $\LS (\K) = \aleph_0$ and $\K_{\le \aleph_0}$; Shelah calls this notion $\PC_{\aleph_0}$ \cite[Definition I.1.4]{shelahaecbook}, and it is also called ``$\PC_{\delta}$ over $\Ll_{\omega_1, \omega}$'' in other places.  We call an AEC $\K$ \emph{Borel presented} (or just \emph{Borel})
 if $\LS (\K) = \aleph_0$ and $\K_{\le \aleph_0}$ is Borel (note that we think of each model as a real, and we also ask for the ordering relation on the countable models to be Borel). 

\begin{cor}\label{borel-cor}
  If $\K$ is an AEC with intersections, $\LS (\K) = \aleph_0$, and $\Ii (\K, \aleph_0) \le \aleph_0$, then $\K$ has a Borel functorial expansion (so in particular, it is analytically presented). 
\end{cor}
\begin{proof}
  Let $\K^+$ be as described by Theorem \ref{aec-to-logic}. By Lemma \ref{embed-dk}, $|\Dd (\K)| \le \aleph_0$. Since $\K^+$ is a functorial expansion, $|\Dd (\K^+)| \le \aleph_0$. This means that the sentences given by Theorem \ref{aec-to-logic} all have countable length, and the result directly follows (use Scott sentences).
\end{proof}

The following stronger property is interesting: let us call an AEC $\K$ \emph{model-complete} if for any $M, N \in \K$, $M \lea N$ if and only if $M \subseteq N$. When does an AEC have a model-complete functorial expansion (in a finitary vocabulary)? This cannot happen for all AECs with intersections: consider the following simple example of Kueker \cite[Example 2.10]{kueker2008}: The vocabulary has a unary predicate $P$, and $K$ is the class of $M$ such that $P^M$ is countably infinite and $|M| \backslash P^M$ is infinite. Order the class by $M \lea N$ if and only if $M \subseteq N$ and $P^M = P^N$. This has intersections but no model-complete functorial expansions. More generally, the AECs which have model-complete functorial expansions are exactly those with finite character.  Finite character is a key property of finitary AECs introduced by Hyttinen and Kes\"{a}l\"{a} \cite{finitary-aec} and isolated by Kueker \cite[Definition 3.1]{kueker2008}. For this result, we assume the reader is familiar with Galois (orbital) types and the definition of finite character:

\begin{thm}\label{fcmc-thm}
  Let $\K$ be an AEC. The following are equivalent:

  \begin{enumerate}
  \item $\K$ has finite character.
  \item $\K$ has a model-complete functorial expansion in a finitary vocabulary.
  \end{enumerate}
\end{thm}
\begin{proof}
  If $\K$ has finite character, let $\K^+$ be the $(<\aleph_0)$-Galois Morleyization of $\K$ \cite[Definition 3.3]{sv-infinitary-stability-afml}: it is obtained by adding a relation symbol for each Galois type of a finite sequence over the empty set. Then it is easy to see that $\K^+$ is model-complete. Conversely, if $\K$ has a model-complete functorial expansion $\K^+$ then since Galois types preserve $\tau (\K^+)$-quantifier-free types, $\K$ must have finite character.
\end{proof}


\section{Equivalence vs. Axiomatization}\label{eva-sec}
    
The previous section showed how to axiomatize an AEC with intersections after functorial expansion.  This is partial progress towards a more general question about the axiomatizability of AECs.

\begin{question}\label{aeclog-quest}
Is there a \emph{natural} logic $\cL^{AEC}$ along with natural notions of fragments $\cF \subseteq \cL^{AEC}$ and elementarity $\preceq_\cF^{AEC}$ that is ``the right logic for AECs'' in the sense that the logic is both:
\begin{enumerate}
	\item \label{one}\underline{Limited:} 
          given a theory $T \subseteq \cL^{AEC}$ and a fragment $\cF \subseteq \cL^{AEC}$ containing $T$, the class
	$$\left( \Mod (T), \preceq_{\cF}^{AEC}\right)$$
	is an AEC (preferably with a bound on the L\"{o}wenheim-Skolem number easily computable from $\cF$ and $T$); and
	\item \underline{Complete:} given an AEC $\left(\K, \lea\right)$, there is a theory $T_\K \subseteq \cL^{AEC}$ and a fragment $\cF_\K \subseteq \cL^{AEC}$ containing $T_\K$ such that
	$$\left(\K, \lea\right) = \left( \Mod (T_\K), \preceq_{\cF_\K}^{AEC}\right)$$
\end{enumerate}
\end{question}

The adjective ``natural'' is intended to forbid direct reference to AECs in defining the logic. 
For instance, the following should be forbidden:  Given a language $\tau$, let $\bL(\tau)$ consist of a sentence $\phi_{\K}$ for each AEC $\K$ in the language $\tau$.  Then satisfaction is defined so $M \vDash \phi_{\K}$ if and only if $M \in \K$.  Clearly, by definition, $\K = \Mod (\phi_\K)$.  This could be similarly extended to a notion of elementarity that would capture $\lea$.  However, this logic is defined by explicit reference to AECs, so seems unnatural and, moreover, unhelpful to adding to the understanding of AECs. Shelah \cite[Definition IV.1.9]{shelahaecbook} defines an extension of $\bL_{\infty, \LS(\K)^+}$ that achieves completeness without a functorial expansion, but similarly fails to be limited (item (\ref{one}) above).  Baldwin and the first author \cite[Theorem 3.9]{baldwin-boney} show that every AEC $\K$ admits a functorial expansion to an $\bL_{\infty, \LS(\K)^+}$ axiomatizable class.  However, not every $\bL_{\infty, \LS(\K)^+}$ axiomatizable class is an AEC (e.g.\ complete metric spaces), so this fails the limited condition.  Between these two examples, it is not known if every AEC is axiomatizable by a sentence of $\bL_{\infty, \infty}$.

A related question has been posed by Makkai and Rosick\'{y}.  Recall a functor between two categories $F:\cC \to \cD$ is an \emph{equivalence} if it is full, faithful, and essentially surjective.

\begin{question}[Makkai, Rosick\'{y}]\label{mr-quest}
Is every AEC equivalent to the models of a sentence of $\Ll_{\infty, \omega}$?
\end{question}

As evidence for this, recall that Makkai and Par\'{e} \cite[Proposition 3.2.3 and Theorem 4.3.2]{makkai-pare} have shown that every accessible category (one that is closed under sufficiently directed colimits and generated by a set of `small' models) is equivalent to the models of an $\bL_{\infty, \infty}$ axiomatizable class; such classes are clearly accessible.  Further discussion of these issues can be found in  \cite[\S4]{multipres-v4-toappear}, \cite{beke-rosicky}, or \cite{lieberman-categ}. After the initial submission of the present paper, Simon Henry \cite{henry-aec-uncountable-v1} gave a topos-theoretic proof showing that the AEC of all uncountable sets (ordered by subset) is a counterexample to Question \ref{mr-quest}.

Although Questions \ref{aeclog-quest} and \ref{mr-quest} seem very similar, the following example illustrates the key difference: an equivalence of categories is {\bf not} required to preserve the underlying structure of sets.  This means that entire models can be condensed to single points.  As an example, consider again quasiminimal classes \cite[Section 5]{quasimin}. First, note that most quasiminimal classes are not finitary: Kirby \cite[Section 2.8]{kirby-note-axiom} proves this for Zilber's pseudominimal fields, but this can also easily be seen for the ``toy example'' of equivalence relations with countably infinite classes, ordered by ``equivalence classes do not grow''.  By Theorem \ref{fcmc-thm} above, such classes do not have functorial expansions that are model complete.  On the other hand, any $\bL_{\infty, \omega}$ axiomatizable class has such a functorial expansion.  Thus we get the following corollary.

\begin{cor}
Zilber's pseudoexponential fields do not have a functorial expansion that is axiomatizable by $\bL_{\infty, \omega}$.
\end{cor}

However, quasiminimal classes are \emph{finitely accessible}.  This is a categorical formulation of the fact that they are closed under directed colimits and every structure is the directed colimit of the finite dimensional models.  Thus, by Makkai and Par\'{e}, the class is equivalent as a category to a $\bL_{\infty, \omega}$ axiomatized class.  Indeed, while Makkai and Par\'{e} work through sketches, there is a straightforward description of the equivalent class.

Fix $\K$ to be a quasiminimal class and, for each $n<\omega$, let $M_n \in \K$ be an $n$-dimensional model.  Then define $T_\K \in \bL_{\infty, \omega}(\tau_\K)$ as follows:
\begin{itemize}
	\item $\tau_\K$ has a sort $S_n$ for each $n<\omega$ and a function $\bar{f}:S_m \to S_n$ for each $\K$-embedding $f:M_n \to M_m$.
	\item $T_\K$ consists of the following:
	\begin{itemize}
		\item if $f:M_n \to M_m$ and $g:M_m\to M_k$ are $\K$-embeddings, then include
		$$\forall x \in S_k \left(\overline{g \circ f}(x) = \bar{f}\left(\bar{g}(x)\right)\right)$$
		\item for each $n, m<\omega$, include
		$$\forall x \in S_n, y \in S_m \exists z \in S_{n+m} \bigvee_{f:M_n \to M_{n+m}, g:M_m\to M_{n+m}} (\bar{f}(z) = x \wedge \bar{g}(z)=y)$$
	\end{itemize}
\end{itemize}

The first sentence says that a model is essentially a functor from the opposite category of the finite dimensional models to the category of sets and the second sentence gives the directedness of the resulting system.  Then the equivalence $F: \K \to \Mod (T_\K)$ is given by, for $M \in \K$
\begin{enumerate}
	\item $S_n^{FM} := \{x:M_n \to M \mid x \text{ is a }\K \text{-embedding}\}$
	\item $\bar{f}^{FM}$ takes $x:M_m \to M$ to $x\circ f: M_n \to M$
\end{enumerate}
Thus, although Zilber's pseudoexponential fields are equivalent to an $\bL_{\infty, \omega}$ axiomatizable class, the equivalence turns every finite dimensional substructure into a point.

We close with an example of an AEC that is well-behaved, yet the authors know of no limited logic (in the sense of Question \ref{aeclog-quest}) that captures it, either up to equivalence of categories or axiomatization of a functorial expansion.

\begin{example}
Let $T$ be a superstable first-order theory and let $\lambda$ be an infinite cardinal. Set $\K^\lambda$ to be the $\lambda$-saturated models of $T$ ordered by elementary substructure.  Since $T$ is superstable, the class is closed under increasing union and, thus, in an AEC with L\"{o}wenheim-Skolem number at most $2^{|T|} + \lambda$.  Since every model of $T$ can be extended to a $\lambda$-saturated one, this class inherits many desirable AEC properties from $\Mod (T)$: amalgamation, tameness, etc. (if $\lambda = \aleph_0$, the class is even closed under regular ultraproducts).
\end{example}

\bibliographystyle{amsalpha}
\bibliography{logic-intersection}

\end{document}